\documentclass{article}
\usepackage[letterpaper, margin=1in]{geometry}
\usepackage{amsmath}
\usepackage{amssymb}
\usepackage{listings}
\usepackage{courier}
\usepackage{mdframed}
\usepackage{fancyhdr}
\usepackage{graphicx}
\graphicspath{{images/}}
\usepackage[utf8]{inputenc}
\usepackage[english]{babel}
\usepackage{amsthm}
\usepackage[T1]{fontenc}

\theoremstyle{plain}
\newtheorem{thm}{Theorem}[section]
\newtheorem{lem}[thm]{Lemma}
\newtheorem{prop}[thm]{Proposition}
\newtheorem{cor}[thm]{Corollary}
\newtheorem{notation}[thm]{Notation}
\newtheorem{conjecture}[thm]{Conjecture}

\theoremstyle{defn}
\newtheorem{defn}[thm]{Definition}

\theoremstyle{remark}

\usepackage{color}

\usepackage{color,soul}
\usepackage{enumitem}
 
\begin{document}

\begin{titlepage}
    \begin{center}
        \vspace*{1cm}
        \Huge
        \textbf{Elementary Results on Forbidden Minors}
        \vspace{3.0cm}
        \Large
        
        \textbf{Arnold Tan Junhan} 
        \vspace{2.0cm}
        \Large
        
        \textbf{Michaelmas 2018 Mini Projects: Graph Theory}
        \vfill
        \vspace{0.8cm}
        \Large
        University of Oxford   
    \end{center}
\end{titlepage}
\tableofcontents
\section{Abstract}
\noindent Throughout this report $G$, $H$, $K$, etc. will denote graphs. Our graphs are finite, always have at least one vertex, and do not have loops or multi-edges.
\\

\noindent We first build up some theory to state \textit{Wagner's Theorem}, and then prove it using a well-known result called \textit{Kuratowski's Theorem}. Actually, the two are equivalent, but we will not prove the latter. (For a proof of Kuratowski's Theorem, see Diester (2000).)
\\

\noindent Following this, we establish some connections between the chromatic number of a graph and some of its forbidden minors. The idea is that if we forbid $G$ to have certain graphs as a minor, then the chromatic number of $G$ cannot be too large. Intuitively, this makes sense: if we disallow $G$ from having too many edges, then this makes it easier to colour the graph with fewer colours; we will of course make this precise.
\\

\noindent We end by stating a well-known conjecture that generalises our work.
\section{Wagner's Theorem}

\begin{defn} \label{minor}
A graph $H$ is a \textbf{minor} of a graph $G$ (or \textbf{$G$ has an $H$-minor}) if there is a sequence $G=G_0, G_1, \ldots, G_k=H$ such that, for $i= 1, \ldots, k$, $G_i$ is obtained from $G_{i-1}$ by contracting an edge, or deleting an edge or a vertex.
\end{defn}

\noindent (We allow $G$ to be a minor of itself. Also, in practice, we only care what the minor $H$ is up to isomorphism: if $G$ has a $K$-minor, and $H \cong K$, then we may simply say $G$ has an $H$-minor.)
\\

\noindent There is an equivalent formulation of minors:

\begin{prop} \label{minoralt}
Graph $H$ is a minor of graph $G$ if and only if, writing $V(H) = {v_1, \ldots, v_h}$, there are nonempty disjoint subsets $V_1, \ldots, V_h$ of $V(G)$ such that $G[V_i]$ is connected for each $i$ and $e(V_i,V_j)>0$ whenever $v_iv_j \in E(H)$.
\end{prop}

\begin{defn} \label{subdivision}
\textbf{Subdividing} an edge $xy$ of a graph $H$ deletes the edge $xy$ and adds a new vertex $v$ with neighbours $x$ and $y$. $G$ is a \textbf{subdivision} of $H$ if it can be obtained from $H$ by a sequence of subdivisions of edges -- that is, there is a sequence $H=H_0$,$\ldots$,$H_k=G$ such that for each $1 \leq i \leq k$, $\exists v_i \in V(H_i) \backslash V(H_{i-1}), w_iu_i \in E(H_{i-1})$ such that $H_i=(V(H_{i-1}) \cup \{ v_i \},(E(H_{i-1}) \backslash \{ w_iu_i \} ) \cup \{ v_iw_i, v_i u_i \} )$. $H$ is a \textbf{topological minor} of $G$ if $G$ contains a subdivision of $H$ as a subgraph.
\end{defn}

\noindent Note that $G$ being a subdivision of $H$ means precisely that we can replace each edge $uv$ in $H$ by some path from $u$ to $v$ (adding some vertices to $H$), such that these paths are internally vertex disjoint, and $G$ is the resulting graph.

\begin{notation}
Write $H \leq G$ to mean $G$ contains $H$, $H \leq_m G$ to mean $G$ has an $H$-minor, and $H \leq_t G$ to mean $G$ contains a subdivision of $H$.
\end{notation}


\noindent We will use either equivalent formulation of minors, depending on which is convenient to us. There are some obvious properties of minors and subdivisions.

\begin{lem} \label{trans}
If $H \leq_m G$ and $K \leq_m H$, then $K \leq_m G$.
\end{lem}

\begin{proof}
Suppose that the sequence $G=G_0, G_1, \ldots, G_k=H$ witnesses $H \leq_m G$, and the sequence $H=H_0, H_1, \ldots, H_l=K$ witnesses $K \leq_m H$. Then the sequence $G=G_0, G_1, \ldots, G_k=H=H_0, H_1, \ldots, H_l=K$ witnesses $K \leq_m G$.
\end{proof}

\noindent It is also clear that if $G$ has $H$ as a subgraph, then $G$ has $H$ as a minor. (Remove the vertices in $V(G)\backslash V(H)$ one by one, in any order, to get a sequence of subgraphs that witnesses $H \leq_m G$.) In fact:

\begin{cor} \label{implications}
$H \leq G \Rightarrow H \leq_t G \Rightarrow H \leq_m G$.
\end{cor}

\begin{proof}
Certainly if $H$ is a subgraph of $G$, then $G$ contains a subdivision of $H$, namely $H$ itself. This proves the first implication. For the second, suppose $G$ contains a subdivision, say $H'$, of $H$. That is, $H' \leq G$, and there is a sequence $H=H_0, \ldots, H_k=H'$ witnessing that $H'$ is a subdivision of $H$. It is enough to show that $H$ is a minor of $H'$, for then we would have $H \leq_m H' \leq_m G$, and we can finish by Lemma \ref{trans}. Define a sequence $H'=H'_0, \ldots, H'_k=H$ by setting $H'_j=H_{k-j}$ for each $0 \leq j \leq k$. This sequence witnesses that $H$ is a minor of $H'$, since each graph $H'_{k-i+1}=H_{i-1}$ in the sequence is obtained from the previous graph $H'_{k-i}=H_i$ (where $1 \leq i \leq k$) by an edge contraction. To be pedantic, recalling that $H_i=(V(H_{i-1}) \cup \{ v_i \},(E(H_{i-1}) \backslash \{ w_iu_i \} ) \cup \{ v_iw_i, v_i u_i \} )$, we have

\begin{align*} 
& \ \ \ \ H_i/(v_iw_i) \\
&= ((V(H_{i-1}) \backslash \{ w_i \}) \cup \{w'_i\},(E(H_{i-1}) \backslash \{ \text{edges containing }w_i \}) \cup \{ w'_ix: x\text{ adjacent in }H_i \text{ to } v_i\text{ or }w_i; x \neq v_i,w_i \}) \\
 &= ((V(H_{i-1}) \backslash \{ w_i \}) \cup \{w'_i\},(E(H_{i-1}) \backslash \{ \text{edges containing }w_i \}) \cup \{ w'_ix: x\text{ adjacent in }H_{i-1} \text{ to } w_i\}) \\
 &\cong  (V(H_{i-1}),E(H_{i-1})) \\
 &=  H_{i-1}. \\
\end{align*}

\noindent This completes the proof that $H$ is a minor of $H'$, hence the result holds.
\end{proof}

\begin{cor} \label{q1}
If $G$ is a subdivision of $H$, then $H$ is a minor of $G$.
\end{cor}

\begin{proof}
We have already shown this in the proof of Corollary \ref{implications}, but it also follows from the corollary, since if $G$ is a subdivision of $H$, then $H \leq_t G$, so $H \leq_m G$, so $H$ is a minor of $G$.
\end{proof}

\noindent On the other hand, we might ask, if $H$ is a minor of $G$, must $G$ contain a subdivision of $H$ as a subgraph? The answer, in general, is no. Famously, $K_5$ is a minor of the Petersen graph, but not a topological minor.

\begin{prop}
Let $P$ be the \textbf{Petersen graph}, defined by $$(V(P),E(P))=(\{u_0,\ldots, u_4, v_0,\ldots, v_4\}, \{u_iu_{i+1}, \ u_iv_i, \ v_iv_{i+2} : i=0, \ldots, 4\})$$
where subscripts are read modulo $5$.
\\
$K_5$ is a minor of $P$. but $P$ does not contain a subdivision of $K_5$ as a subgraph.
\end{prop}

\begin{proof}
Note that each $u_i$ is neighbours precisely with $u_{i-1}$, $u_{i+1}$ and $v_i$, and each $v_i$ is neighbours precisely with $v_{i-2}$, $v_{i+2}$ and $u_i$. In either case, there are exactly three distinct neighbours, since we are reading subscripts modulo $5$. Therefore $P$ is $3$-regular.
\\

\noindent For the first claim, let us use our alternative formulation of minors. To see that $K_5$ is a minor of $P$, consider the disjoint nonempty subsets $V_0=\{ u_0,v_0\}, \ldots, V_4=\{u_4,v_4\}$. Each $G[V_i] = (\{ u_i, v_i \},\{ u_iv_i \})$ is obviously connected, so it remains to show that $e(V_i,V_j) > 0$ whenever $i \neq j$ ($i,j \in \{0,\ldots, 4\}$). If $i,j$ differ modulo $5$ by $1$, then without loss $j=i+1$ so $u_iu_j \in E(P)$ shows that $e(V_i,V_j)>0$. If $i,j$ differ modulo $5$ by $2$, then without loss $j=i+2$ so $v_iv_j \in E(P)$ shows that $e(V_i,V_j)>0$. There are no other possibilities, so we are done.
\\

\noindent Let us prove the second claim. It follows by definition that subdividing an edge of a graph $H$ creates a new vertex of degree $2$, but does not change the degree of any vertex of $H$. Therefore, if $H'$ is a subdivision of $H$ and $\Delta(H) \geq 2$, then $\Delta(H')=\Delta(H)$, and so if $G$ contains a subdivision of $H$, then $\Delta(G) \geq \Delta(H)$. In particular, when $G$ is the Petersen graph $P$, and $H$ is $K_5$, then this says that if the Petersen graph contained a subdivision of $K_5$, then we would have $3 = \Delta(P) \geq \Delta(K_5) = 4$: an obvious contradiction. Therefore $K_5 \nleq_t P$.
\end{proof}

\noindent We might wonder if the result becomes true under some simple additional assumption. This is indeed the case, but before stating the result we give a useful lemma.

\begin{lem} \label{minimal}
Let $G$ and $H$ be graphs, where $V(H)=\{v_1,\ldots,v_h\}$. Let $H$ be a minor of $G$, and let this be witnessed by the nonempty disjoint subsets $V_1,\ldots, V_h$ from Proposition \ref{minoralt}. Suppose that no proper subgraph of $G$ contains $H$ as a minor. Then,

\begin{enumerate}

\item Each $G[V_i]$ is minimally connected, i.e., a tree.
\item Whenever $v_iv_j \in E(H)$, we have $e(V_i,V_j)=1$.
\item Whenever $v_iv_j \notin E(H)$ then $e(V_i,V_j)=0$.
\item For every leaf $w$ of a tree $G[V_i]$ of size larger than $1$, we can find some $j \neq i$ such that $e(\{w\}, V_j)>0$.
\item $G[V_i]$ has at most $d(v_i)$ leaves.
\item $V_1,\ldots, V_h$ cover $V(G)$.
\end{enumerate}
\end{lem}

\begin{proof}
If the first claim were false, then we could remove an edge from $V_i$ to get a proper subgraph $G_1$ of $G$, where $G_1[V_i]$ is still connected, so the same subsets $V_1,\ldots, V_h$ would witness that the proper subgraph $G_1$ has an $H$-minor.
\\

\noindent If the second claim were false, then there would be at least two edges from some $V_i$ to some $V_j$, and if we remove one of these we get a proper subgraph $G_2$ of $G$ where $G[V_i]=G_2[V_i]$ for each i, and there is still an edge in $G_2$ from $V_i$ to $V_j$ whenever $v_iv_j \in E(H)$. Hence the same subsets $V_1,\ldots, V_h$ would witness that the proper subgraph $G_2$ has an $H$-minor.
\\

\noindent If the third claim were false, then we could find $v_iv_j \notin E(H)$ with an edge between $V_i$ and $V_j$. Removing this edge gives a proper subgraph $G_3$ of $G$ where $G[V_i]=G_3[V_i]$ for each i. Hence the same subsets $V_1,\ldots, V_h$ would witness that the proper subgraph $G_2$ has an $H$-minor.
\\

\noindent If the fourth claim were false, then we could remove $w$ to get a tree $G[V_i] -w =G[V_i \backslash \{w\} ]$. This is still nonempty and connected, and there is still an edge in the proper subgraph $G_4=G-w$ of $G$ from $V_i \backslash \{w\}$ to $V_j$ whenever $v_iv_j \in E(H)$. Hence the subsets $V_1,\ldots, V_i\backslash \{w\}, \ldots, V_h$ would witness that the proper subgraph $G_4=G-w$ has an $H$-minor.)
\\

\noindent We will use the claims above to prove the fifth claim. Without loss $G[V_i]$ has more than one vertex (otherwise $G[V_i]$ has a single vertex, of degree zero -- this is not a leaf). There is an injective function from the set of leaves of $G[V_i]$ to the set $\{j \in \{1,\ldots, h\}: v_j \text{ is a neighbour of }v_i\}$, given by sending a leaf $w$ to any fixed choice of $j$ given by the fourth claim: since $e(\{w\}, V_j)>0$, by the third claim $v_j$ is a neighbor of $v_i$, so this function is well-defined. It is an injection since if two leaves $w,w'$ are sent to the same $j$, then there are edges from $w$ into $V_j$ and from $w'$ into $V_j$. By the second claim, $w=w'$.
\\

\noindent If the sixth claim were false, then $V_1,\ldots, V_h$ would witness that the proper subgraph $G[V_1 \cup \ldots \cup V_h]$ of $G$ has an $H$-minor.
\end{proof}

\begin{prop} \label{converse}
Suppose $H$ is a minor of $G$ and $\Delta(H) \leq 3$. Then $G$ contains a subdivision of $H$.
\end{prop}

\begin{proof}
We prove by induction on $n=e(G)+v(G)$ that for any $H$ with $H \leq_m G$ and $\Delta(H) \leq 3$, we have $H \leq_t G$.
If $n=1$, then $V(G)=1$ so the result is trivial: $H \leq_m G$ if and only if $H \leq_t G$ if and only if $H=G$.
\\
Let $n=e(G)+v(G)>1$, and suppose the result holds for all smaller values of $n$ -- that is, whenever we have graphs $H$ and $G'$, where $\Delta(H) \leq 3$, $H \leq_m G'$, and $e(G')+v(G')<n$, then $H \leq_t G'$. The condition $e(G')+v(G')<n$ holds whenever $G'$ is a proper subgraph of $G$. Therefore, it suffices to prove the following claim:
$$\textit{If }G\textit{ is such that } H \leq_m G \textit{, but no proper subgraph of } G \textit{ has an } H \textit{-minor, where } \Delta(H) \leq 3 \textit{, then } H \leq_t G.$$

\noindent (Once we have the claim, then for general $G$, if a proper subgraph $G'$ contains an $H$-minor, then we may apply the inductive hypothesis to $G'$ to deduce that $H \leq_t G'$. Hence $G$ contains $G'$, which contains a subdivision of $H$, so $G$ certainly contains a subdivision of $H$.)
\\

\noindent Let us prove the claim. Writing $V(H)=\{v_1, \ldots, v_h\}$, there are disjoint nonempty subsets $V_1,\ldots, V_h$ of $V(G)$ such that each $G[V_i]$ is connected, and $e(V_i,V_j)>0$ whenever $v_iv_j \in E(H)$. By the lemma above, each $G[V_i]$ is \textit{minimally} connected, and $e(V_i,V_j)=1$ whenever $v_iv_j \in E(H)$.
\\

\noindent If a tree $G[V_k]$ has a single vertex $w_k$, then we can find some $j \neq k$ such that $e(\{w_k\}, V_j)>0$ if (and only if, by the lemma above) $v_k$ is not isolated in $H$. Moreover, for every leaf $w$ of a tree $G[V_i]$ of size larger than $1$, we can find some $j \neq i$ such that $e(\{w\}, V_j)>0$.
\\

\noindent Denote by $w_{ij}$ the unique leaf $w \in G[V_i]$ with $e(\{w\}, V_j)>0$ (if it exists). In this case, fixing $i$ and running over all neighbours $v_j$ of $v_i$ in $H$ gives us a collection of (not necessarily distinct) leaves $w_{ij}$ (where $j$ varies) in $G[V_i]$. By the lemma above, this collection has size at most $3$, since $\Delta(H) \leq 3$. This means the tree $G[V_i]$ has at most $3$ leaves. In particular there is always a vertex $x_i$ of $G[V_i]$, and paths from $x_i$ to each $w_{ij}$ that are pairwise disjoint except at $x_i$. Indeed, if there is only one leaf $w$ in $G[V_i]$, take $x_i=w$, and the trivial path from $x_i$ to itself. If there are only two leaves, let $x_i$ be one of these, and take our two paths to be the trivial path from $x_i$ to itself, and a path from $x_i$ to the other leaf. (Actually, it is unique, since $G[V_i]$ is a tree.) If there are three leaves $w_{ij_1}$, $w_{ij_2}$, $w_{ij_3}$, first consider a path from $w_{ij_1}$ to $w_{ij_3}$, and a path from $w_{ij_2}$ to $w_{ij_3}$. The last edge in both of these paths has to be the unique edge $xw_{ij_3}$ that meets $w_{ij_3}$. Let us take $x_i$ to be the first vertex in the path from $w_{ij_1}$ to $w_{ij_3}$ that is also in the path from $w_{ij_2}$ to $w_{ij_3}$, and take our three paths to be the three unique paths starting at $x_i$ and ending at each respective leaf. By construction, none of these paths intersect other than at $x_i$.
\\

\noindent In the special case where $G[V_k]$ has a single vertex $w_k$ and $e(\{w_k\}, V_j)=0$ for each $j \neq k$, then $v_k$ is isolated in $H$; we can still define $x_k=w_k$.

\noindent This exhibits, as desired, a subdivision of $H$ as a subgraph $H'$ of $G$. Explicitly, take $H'$ to be the union of subgraphs of $G$ of the form:
\begin{itemize}
    \item paths $Q_{ij}=P_{x_iw_{ij}}(w_{ij}w_{ji})P_{w_{ji}x_j}$, where $v_iv_j \in E(H)$, $P_{x_iw_{ij}}$ is the unique path in $G[V_i]$ from $x_i$ to $w_{ij}$, $w_{ij}w_{ji}$ is an edge in $G$, and $P_{w_{ji}x_j}$ is the unique path in $G[V_j]$ from $x_j$ to $w_{ji}$, and we have concatenated them together to get a path $Q_{ij}$ in $G$;
    \item trees $G[V_k]$ on a single vertex $w_k$, such that $e(\{w_k\}, V_j)=0$ for each $j$.
\end{itemize}

\noindent Note that the collection of paths $Q_{ij}$ are internally vertex disjoint.
\noindent It remains to see that $H'$ is a subdivision of $H$. This is clear, if you let the vertices $x_1,\ldots, x_h$ in $H'$ correspond to the vertices $v_1, \ldots, v_h$ of $H$. Then, whenever $v_iv_j \in E(H)$, the path $Q_{ij}$ from $x_i$ to $x_j$ corresponds to a repeated subdivision of the edge $v_iv_j$.
\end{proof}

\begin{lem} \label{either}
If $K_5$ is a minor of $G$, then $G$ contains a subdivision of either $K_{3,3}$ or $K_5$.
\end{lem}

\begin{proof}
If $K_5$ is a minor of $G$, let $H$ be a minimal subgraph of $G$ containing a $K_5$-minor -- that is, $H$ contains a $K_5$-minor, but no proper subgraph of $H$ contains a $K_5$-minor. By Proposition \ref{minoralt}, there exist disjoint nonempty subsets $V_1, \ldots, V_5$ such that each $H[V_i]$ is connected, and $e(V_i, V_j)>0$ for distinct $i,j \in \{1,\ldots, 5\}$. By Lemma \ref{minimal}, each $H[V_i]$ is a tree, and $e(V_i,V_j)$ for distinct $i,j$. That is, fixing $i$, there is exactly one edge from $V_i$ to each of the other four $V_j$. Consider the tree $T_i$ obtained from the tree $H[V_i]$ by adding these four edges. ($T_i$ is indeed a tree since, for instance, adding these edges does not create cycles or destroy connectivity.) Note that $T_i$ has exactly $4$ leaves, one in each of the $V_j$ other than $V_i$ itself. (Adding the four edges to $H[V_i]$ created these four leaves, and also increased the degree of any leaf of $H[V_i]$.)
We consider two possibilities for the form of $T_i$.
\\
First, note that a tree $T$ has at least $\Delta(T)$ leaves. (Let $v \in V(T)$ have degree $\Delta(T)$; for each edge $vu$ consider a maximal path starting with $vu$; by maximality this path ends at a leaf. We get $\Delta(T)$ paths starting at $v$ and ending at different leaves, because these paths are necessarily disjoint except at $v$ -- otherwise $T$ would have a cycle.) Therefore $\Delta(T_i) \leq 4$. The handshaking lemma now says $2(3+r_2+r_3+r_4)=4+2r_2+3r_3+4r_4$, where $r_i$ denotes the number of vertices of degree $i$. This simplifies to $2=r_3+2r_4$, so we have the two cases below for $T_i$.

\begin{itemize}
    \item $T_i$ has one vertex $x_i$ of degree $4$, four leaves, possibly some vertices of degree $2$, and no other vertices. (That is, $T_i$ is a subdivision of $K_{1,4}$.)
    \item $T_i$ has two vertices $y^1_i,y^2_i$ of degree $3$, four leaves, possibly some vertices of degree $2$, and no other vertices.
\end{itemize}

\noindent If each of the five $T_i$ fall under the first case, then we have that $H$ is a subdivision of $K_5$. (Consider the vertices $x_1, \ldots, x_5$. Between each two of these there is a path in $H$; these paths are internally vertex disjoint.)
\\

\noindent If some $T_i$ falls under the second case, we claim $H$ has a $K_{3,3}$-minor. Once we have this, then we are done: since $\Delta(K_{3,3})=3$, by Proposition \ref{converse} $H$ (and hence $G$) contains a subdivision of $K_{3,3}$. It remains to show $H$ has a $K_{3,3}$-minor. Well, if we contract $G[V_i]$ onto the two vertices $y^1_i,y^2_i$ of $T_i$ having degree $3$, and contract the other $G[V_j]$ onto singletons, then we get six-vertex graph that contains $K_{3,3}$.

\end{proof}

\begin{thm}[Kuratowski's Theorem] \label{kuratowski}
A graph $G$ is planar if and only if $G$ does not contain a subdivision of $K_5$ or of $K_{3,3}$.
\end{thm}

\noindent For a proof of Kuratowski's Theorem, see Diestel (2000).

\noindent We are now in position to state and prove Wagner's Theorem from Kuratowski's Theorem.

\begin{thm}[Wagner's Theorem] \label{wagner}
A graph $G$ is planar if and only if $G$ contains neither $K_5$ nor $K_{3,3}$ as a minor.
\end{thm}

\begin{proof}
For the forward direction, suppose $G$ is planar. By Kuratowski's Theorem, $G$ does not contain a subdivision of $K_5$ or of $K_{3,3}$. By Lemma \ref{either}, $G$ cannot contain $K_5$ as a minor. Moreover, by Proposition \ref{converse}, $G$ cannot contain $K_{3,3}$ as a minor, otherwise $G$ would contain a subdivision of $K_{3,3}$, because  $\Delta(K_{3,3}) = 3$.
\\

\noindent For the backward direction, suppose $G$ contains neither $K_5$ nor $K_{3,3}$ as a minor. Then, whenever $H$ is a subgraph of $G$, we must have $K_5 \nleq_m H$ and $K_{3,3} \nleq_m H$. (If $K_5 \leq_m H$, then $K_5 \leq_m H \leq G$ would imply that $K_5 \leq_m G$, but this is forbidden. Similarly, we cannot have $K_{3,3} \leq_m H$.) By Corollary \ref{q1}, $H$ is neither a subdivision of $K_5$ nor of $K_{3,3}$.
We have shown that no subgraph $H$ of $G$ is a subdivision of $K_5$ or of $K_{3,3}$. Therefore by Kuratowski's Theorem, $G$ is planar.
\end{proof}
\section{The Chromatic Number and Forbidden Minors}

In this section we prove several results of the form: \textit{if the chromatic number of a graph $G$ is large enough, then this forces $G$ to contain a complete graph as its minor.}

\begin{defn}
Let $x \in V(G)$, and $U \subseteq V(G) \backslash \{x\}$. An \textbf{$x,U$-fan} is a set of paths from $x$ to $U$ that are pairwise disjoint except at $x$. The \textbf{size} of the fan is the number of paths in the collection.
\end{defn}

\begin{lem}[Fan Lemma] \label{fan}
If $G$ is $k$-connected, and the size of $U$ is at least $k$. Then there exists an $x,U$-fan of size $k$.
\end{lem}

\begin{proof}Add a new vertex $u$ to $G$, along with edges from $u$ to each vertex of $U$. This new graph $G'$ is again $k$-connected, since $v(G') \geq k+1$ and no set of size at most $k-1$ separates $G'$. Indeed, suppose there was such a set $S$.

\begin{itemize}
    \item If $S \subseteq V(G)$, then as $G'-S$ is disconnected, hence so is $G-S=G'-u-S = G'-S-u$. Indeed, $G'-S$ has a component $C$ containing $u$, and at least one other component $D$ not containing $u$, so $G'-S-u$ has at least two components, namely D and some component of $C-u$, since $C-u$ is nonempty: it contains some element of $U$, since the size of $S$ was less than $k$.
    \item If $u \in S$, then $G-S=G'-u-S =G'-S$ is disconnected.
\end{itemize}

\noindent Now, $G'$ is $k$-connected, so by Menger's Theorem, there exist $k$ independent paths $P_1, \ldots, P_k$ from $x$ to $u$ in $G'$. (Menger's Theorem says there exist $K_{G'}(x,u)$ independent paths from $x$ to its non-neighbour $u$, but $K_{G'}(x,u) \geq \kappa(G') \geq k$, where the first inequality is because no set of size less than $\kappa(G')$ separates $x$ from $u$.) Each such path is necessarily of the form $x\ldots u_i u $ for some $u_i \in U$, where $u_i \neq u_j$ if $i \neq j$. Therefore the paths $x\ldots u_i$ in $G$ form our desired $x, U$-fan of size $k$.
\end{proof}

\begin{lem} \label{oddcyclesminor}
If a graph $G$ has a cycle, then $G$ contains a $K_3$-minor.
\end{lem}

\begin{proof}
Suppose $G$ has a cycle $C_n$ $(n \geq 3)$ as a subgraph. By Corollary \ref{implications}, $G$ has a $C_n$-minor. Now note that $C_n$ in turn has a $C_3$-minor, as witnessed by the sequence $C_n,C_{n-1}, \ldots, C_3$, each obtained from the previous by contracting an edge. Hence, Lemma \ref{trans} implies that $K_3 \cong C_3$ is a minor of $G$.
\end{proof}

\begin{prop}
If a graph $G$ has $\chi (G) \geq 3$ then $G$ contains a $K_3$-minor.
\end{prop}

\begin{proof}
This is straightforward, by noting some equivalent characterisations of bipartite graphs: a graph $G$ is bipartite if and only if it is $2$-colourable, if and only if it contains no odd cycles. Hence, $\chi (G) \geq 3$ means $G$ is not $2$-colourable, therefore $G$ is not bipartite, therefore $G$ has an odd cycle, therefore $G$ contains a $K_3$-minor by Lemma \ref{oddcyclesminor}.
\end{proof}

\begin{prop}
If a graph $G$ has $\chi (G) \geq 4$ then $G$ contains a $K_4$-minor.
\end{prop}

\begin{proof}
Let us proceed by induction on $v(G)$. If $v(G) \leq 3$, there is nothing to show; if $v(G)=4$ and $\chi(G) \geq 4$ then $\chi(G) = 4$, so $G$ must be $K_4$. Now suppose $v(G)>4$, and that the result holds for all graphs on fewer vertices. Without loss, we may assume $G$ is connected. (Indeed, $\chi(G)$ is the maximum of the chromatic numbers of the components of $G$, so $G$ has a component $C$ with $\chi(C) \geq 4$. If $G$ is not connected then $C$ is a proper subgraph of $G$, so by induction $C$ -- and hence $G$ -- must have a $K_4$-minor.) Now we split into four cases.
\begin{itemize}

    \item Suppose $\kappa(G)=0$. $G$ is not $1$-connected, so as $v(G) \geq \chi(G) \geq 4 \geq 2$, we have that $G$ is disconnected. This contradicts our assumption.
    
    \item Suppose $\kappa(G)=1$. $G$ is not $2$-connected, so as $v(G) \geq \chi(G) \geq 4 \geq 
    3$, we have that some set $\{c\} \subseteq V(G)$ separates $G$. $G-c$ will have a partition into two disjoint subgraphs $G'_1$ and $G'_2$ (these will be unions of the components of $G-c$). Let $G_i=G[V(G'_i) \cup \{c\}]$, for $i=1,2$. Then $V(G_1) \cap V(G_2)=\{c\}$, and $G_1[\{c\}]=G_2[\{c\}]$ are complete (they are both equal to $(\{c\},\emptyset) \cong K_1)$, so $\chi(G)=\chi(G_1 \cup G_2)=\text{max}\{\chi(G_1),\chi(G_2)\}$. Without loss, $\chi(G_1)=\chi(G) \geq 4$. $G_1$ is a proper subgraph of $G$, as $\emptyset \neq V(G'_2) \subseteq V(G) \backslash V(G_1)$), so by induction $G_1$ (and hence $G$) must have a $K_4$-minor.
    
    \item Suppose $\kappa(G)=2$, and $\{x,y\}$ is a separating set: there are subgraphs $G_1$ and $G_2$ such that $G=G_1 \cup G_2$, $V(G_1)\cap V(G_2)=\{x,y\}$, $e(V(G_1)\backslash \{x,y\},V(G_2)\backslash \{x,y\})=0$ and $V(G_i)\backslash \{x,y\} \neq \emptyset$ for both $i$. Suppose for a contradiction that both $G_1+xy$ and $G_2+xy$ can be $3$-coloured. Both colourings must colour $x$ and $y$ with different colours, so after possibly permuting one of the colourings, we can take their union to get a $3$-colouring on $G$; this is a contradiction. Hence, assume without loss that $\chi(G_1+xy) \geq 4$. $G_1+xy$ has fewer vertices than $G$ (as $V(G_2)\backslash \{x,y\} \neq \emptyset$), so by induction it has a $K_4$-minor. We are not quite done yet, as $xy$ need not be an edge of $G$. If it is, we are done, so suppose not. Note that there is a path from $x$ to $y$ in $G_2$ -- without loss $x$ has a neighbour $z$ in $G_2$, but $x$ cannot be a cutvertex, so removing it shows that there is a path from $z$ to $y$ that does not use $x$; concatenating $xz$ with a minimal such path produces a path $P_{xy}$ from $x$ to $y$ in $G_2$. This path 'substitutes' for the missing edge $xy$; we can just contract it and pretend as though we have the edge $xy$ in our graph. In other words, $K_4 \leq_m G_1+xy \leq_m G_1 \cup P_{xy} \leq G$ shows that $G$ has a $K_4$-minor.

    \item Finally, suppose $\kappa(G)\geq 3$. Let $v \in V(G)$. Then $\kappa(G-v) \geq \kappa(G)-1 \geq 2$, so $G-v$ is $2$-connected, and in particular contains a cycle. (Otherwise $G-v$ is connected and acyclic, i.e., a tree. It has a leaf $l$; if we delete the unique neighbour $n$ of $l$ from this tree we get a disconnected graph, because $n$ is not another leaf, since $G-v$ has at least $3$ vertices. This contradicts that $G-v$ is $2$-connected.) Take $U$ to be the set of vertices of this cycle, so its size is at least $3$. Since $G$ is $3$-connected, by Lemma \ref{fan} there exist three paths $P_1,P_2,P_3$ from $v$ to this cycle, that are disjoint except at $v$. Let $v_1,v_2,v_3$ be the three distinct end-vertices of the $P_i$. Then the sets $\{v\}, \{v_1\},\{v_2\},\{v_3\}$ witness that $G$ has a $K_4$-minor, since we can contract the cycle onto the three vertices $v_1,v_2,v_3$, and contract the paths $P_i$ to edges -- in the resulting graph, $x$ is neighbours with each of the three vertices in a triangle. This is a copy of $K_4$.

\end{itemize}
\end{proof}

\noindent In light of the previous two propositions, we might conjecture:

\begin{conjecture}\label{hadwiger}
For every $k \geq 2$, if $\chi (G) \geq k$ then $G$ contains a $K_k$-minor.
\end{conjecture}

\noindent This is precisely the famous \textit{Hadwiger's Conjecture}.
\\
Let us prove a weaker statement:

\begin{thm}
For every $k \geq 2$, there exists a constant $c(k)$ such that if $\chi (G) \geq c(k)$ then $G$ contains a $K_k$-minor.
\end{thm}
\begin{proof}
See proof of the restated version, Theorem \ref{restate}.
\end{proof}

\begin{notation}
For a graph $G$, write $d_G(x,y)$ for the length of a shortest path between vertices $x$ and $y$ (if it exists). This is the \textbf{distance} from $x$ to $y$. Let $G$ be a connected graph (so the distance between any two vertices is well-defined), and $x$ be a vertex of $G$. Define $S_d(x)= \{y \in $V(G)$: d_G(x,y)=d\}$ for each $d \geq 0$. For instance, $S_0(x)=\{x\}$.
\end{notation}

\noindent The idea is that if we fix a vertex $x$, then the sets $S_d(x)$ partition $V(G)$ into 'layers' around $x$, where each vertex $y$ in an outer layer has some edge entering the layer below it, and there are no edges between any two layers that are not adjacent. Let us state this formally.

\begin{lem}\label{layers}
Let $G$ be a connected graph, and $x$ a vertex of $G$. If $d > 0$ and $y \in S_d(x)$, then $\exists z \in S_{d-1}(x)$ such that $zy \in E(G)$. Furthermore, whenever $0 \leq c < d-1$, we have $e(S_{c}(x),S_{d}(x))= 0$.
\end{lem}

\begin{proof}
For the first claim, note that $y \in S_d(x)$ implies there is a path of length $d$ from $x$ to $y$. Let this path be $v_0\ldots v_{d}$ with $v_0=x$, $v_d=y$. Consider the path $v_0\ldots v_{d-1}$ in $G$ from $x$ to $v_{d-1}$, of length $d-1$. This must be a shortest path from $x$ to $v_{d-1}$, for if there were a path of length less than length $d-1$ from $x$ to $v_{d-1}$, then concatenating it with the edge $v_{d-1} v_d$ would result in a path of length less than $d$ from $x$ to $y$, then we would have $d_G(x,y)<d$, contradicting $y \in S_d(x)$. Therefore $v_{d-1} \in S_{d-1}(x)$. Since $v_{d-1}v_d \in E(G)$, this proves the first claim.
\\
For the second claim, suppose for a contradiction that there exists some $0 \leq c < d-1$ and some $w \in S_{c}(x)$, $y \in S_{d}(x)$ such that $wy \in E(G)$. Let $w_0 \ldots w_c$ be a shortest path from $x$ to $w$, where $w_0=x, w_c=w$. Then, concatenating with the edge $w_c y$, we have a path $w_0 \ldots w_c y$ from $x$ to $y$ of length $c+1 < d $. This shows that $d_G(x,y)<d$, contradicting $y \in S_{d}(x)$.
\end{proof}

\begin{lem}
Let $G$ be a connected graph, and $x$ a vertex of $G$. Then there is some $D \geq 0$ such that $S_d(x)$ is empty whenever $d > D$, and nonempty whenever $0 \leq d \leq D$.
\end{lem}

\begin{proof}
First observe that $S_d(x)$ is always empty for large enough $d$. For instance, take $D'$ to be the length of a longest path in $G$, then $S_d(x)$ is empty whenever $d > D'$. Then $S_d(x)$ can only be nonempty for a finite number of values $d=0,1,\ldots, D'$. Let $D$ be the largest among these such that $S_D(x)$ is nonempty. Then, for all $0 \leq d \leq D$, $S_d(x)$ will also be nonempty. Indeed, suppose not, then there is $d \in \{0,\ldots, D-1\}$ such that $S_d(x) = \emptyset$. Let us assume $d$ is maximal as such, so $\exists y \in S_{d+1}(x)$. The lemma above implies, in particular, that there is some $z \in S_{d}(x)$. This contradicts $S_d(x) = \emptyset$. Therefore we must indeed conclude that for all $0 \leq d \leq D$, $S_d(x)$ is nonempty.
\end{proof}

\begin{prop} \label{maxcol}
Let $G$ be a connected graph, and $x$ a vertex of $G$. Then there is some $d \geq 0$ such that $\chi (G[S_d(x)]) \geq \chi (G) /2$.
\end{prop}

\begin{proof}
Let $D$ be as in the lemma above. Choose $d \in \{0,\ldots, D\}$ which maximizes $\chi (G[S_d(x)])$. Let $\chi =\chi (G[S_d(x)])$. It is enough to show that $\chi \geq \chi (G) /2$, or equivalently, that $G$ is $2\chi$-colourable. Let us $2\chi$-colour $G$ starting with the outermost layer $G[S_D(x)]$ and working inwards. By the maximality assumption, we can certainly (properly) $\chi$-colour each layer $G[S_b(x)]$ (where $0 \leq b \leq D$). Therefore we can colour $G[S_D(x)]$ with colours from $\{1,\ldots,\chi\}$, $G[S_{D-1}(x)]$ with colours from $\{\chi+1,\ldots,2\chi\}$, $G[S_{D-2}(x)]$ with colours from $\{1,\ldots,\chi\}$, and alternating so on and so forth, until we reach $G[S_0(x)]={x}$ (which gets a single colour, of course). This procedure colours all of $G$, and indeed gives us a proper colouring on $G$. (Edges within a single layer $G[S_b(x)]$ do not cause problems, since this colouring on $G$ arises from properly colouring each layer; on the other hand an edge between layers can only exist between adjacent layers by Lemma \ref{layers}, and such an edge can never receive the same colour on both its edges, since we are using entirely different colour schemes on alternating layers!)
\end{proof}

\begin{thm} \label{restate}
For every $k \geq 2$, if $\chi (G) \geq 2^k$ then $G$ contains a $K_k$-minor.
\end{thm}
\begin{proof}
We will proceed by induction on $k$.
The base case is $k=2$. In this case, if $\chi (G) \geq 4$ then in particular, since $\chi(G)>1$, $G$ has an edge, so $G$ certainly contains a $K_2$-minor.
\\
Now assume $k > 2$, and suppose the result holds for the case $k-1$. That is, suppose that whenever $H$ is any graph with $\chi(H) \geq 2^{k-1}$, then $H$ contains a $K_{k-1}$-minor. We must show that if $\chi (G) \geq 2^k$ then $G$ contains a $K_k$-minor. Without loss of generality, we may assume $G$ is connected. (For general $G$, $\chi (G)$ is the maximum of the chromatic numbers of its components. Therefore if $\chi(G) \geq 2^k$, then $G$ has a component $C$ with $\chi(C) \geq 2^k$. If the result holds for connected graphs, then $C$ contains a $K_k$-minor, hence so does $G$.)
\\

\noindent Consider the induced subgraph $H=G[S_d(x)]$, where $\chi (G) \geq 2^k$, $x$ is any fixed vertex of $G$ and $d \geq 0 $ is as in Proposition \ref{maxcol}. Then $\chi (H)=\chi (G[S_d(x)]) \geq \chi (G) /2 \geq 2^{k-1}$, so by the inductive hypothesis applied to $H$, $H$ contains a $K_{k-1}$-minor.
\\

\noindent Now it is not hard to show that $G$ contains a $K_k$-minor -- already the layer $G[S_d(x)]$ has a $K_{k-1}$-minor, but by Lemma \ref{layers} there is a path from each vertex of $G[S_d(x)]$ to $x$. We can contract all but one edge in each such path, and finish by noting that adding a new vertex to $K_{k-1}$ adjacent to all the other vertices produces $K_k$.

\noindent Let us argue in more detail, using our alternative formulation of minors. First note that $d>0$. (Indeed, if $d = 0 $, then $1=\chi (G[{x}])=\chi (G[S_0(x)]) \geq \chi (G) /2$ shows that $2 \geq \chi (G)\geq 2^k >2$, a contradiction.) Let $V(K_{k-1})= \{ v_1, \ldots, v_{k-1} \}$. Since $K_{k-1} \leq_m H$, there are disjoint nonempty subsets $V_1$, \ldots, $V_{k-1}$ of $V(H)=S_d(x)$ such that each $H[V_i]$ is connected, and $e(V_i,V_j)>0$ whenever $v_iv_j \in E(K_{k-1})$.
\\

\noindent Consider the graph $K_k \cong (V(K_{k-1})\cup \{ x_0 \}, E(K_{k-1}) \cup \{ v_1x_0,\ldots , v_{k-1}x_0 \})$. This is a copy of $K_k$ obtained by adding a new vertex $x_0$ to $K_{k-1}$, adjacent to every other vertex. I show this a minor of $G$. Well, consider the nonempty subsets $V_1, \ldots, V_{k-1}, V_k$ of $V(G)$, where $V_k:= \bigcup_{b=0}^{d-1}S_b(x) \ni x$. These sets witness that $G$ has a $K_k$-minor:

\begin{itemize}

\item They are pairwise disjoint, because $V_1$, \ldots, $V_{k-1}$ were already pairwise disjoint, and $\bigcup_{i=1}^{k-1}V_i \subseteq V(H)=S_d(x) \subseteq V(G) \backslash \bigcup_{b=0}^{d-1}S_b(x) = V(G) \backslash V_k$.

\item Each $G[V_i]$ is connected: for each $1 \leq i \leq k-1$, $G[V_i]=H[V_i]$, which is connected by assumption. (Both these graphs have vertex set $V_i$ and edge set $E(G) \cap V_i^{(2)} = E(G) \cap V_i^{(2)} \cap V(H)^{(2)} = E(H) \cap V_i^{(2)}$.) $G[V_k]$ is connected since for all $u,w \in V(G[V_k])=V_k$, there are paths in $G$ from $u$ to $x$ and from $x$ to $w$, which concantenated produce a path in $G$ from $u$ to $w$. To be clear, say $u \in S_b(x)$, where $0 \leq b \leq d-1$. If $b=0$ then there is a trivial path from $u$ to $x$, of length $0$. Otherwise, apply Lemma \ref{layers} repeatedly to see that $\exists u_1 \in S_{b-1}(x)$ such that $u_1u \in E(G)$, $\exists u_2 \in S_{b-2}(x)$ such that $u_2u_1 \in E(G)$, $\ldots$, $\exists u_b \in S_{0}(x)$ such that $u_bu_{b-1} \in E(G)$. This yields a path $u u_1 \ldots u_b$ from $u$ to $x$. Similarly, there is a path from $w$ to $x$.

\item $e(V_i,V_j)>0$ whenever $v_iv_j \in E(K_k)$: this holds if $v_iv_j$ is an edge in the copy of $K_{k-1}$ we started with, so suppose we have an edge $v_ix_0 \in E(K_k)$, where $1 \leq i < k$. We must show that $e(V_i,V_k)>0$. First, note that $V_i \neq \emptyset$, so $\exists y \in V_i$. Then Lemma \ref{layers} says $\exists z \in S_{d-1}(x)$ such that $zy \in E(G)$. Since $y \in V_i$ and $z \in V_k$, this shows $e(V_i,V_k)>0$ as required.
\end{itemize}

\noindent This completes the proof that $G$ has a $K_k$-minor.
\end{proof}
\section{Conclusion}

\noindent We have shown that for every $k \geq 2$, there is a constant $c(k)$ such that if $\chi (G) \geq c(k)$ then $G$ contains $K_k$ as a minor. Our choice of $c(k)$ was exponential in $k$, but Kostochka (1984) achieved $c(k)=\mathcal{O}(k \sqrt{log(k)})$.

\noindent Hadwiger's Conjecture says we can do better still:

\begin{conjecture}[Hadwiger's Conjecture] \label{hadwiger}
For every $k \geq 2$, if $\chi (G) \geq k$ then $G$ contains $K_k$ as a minor.
\end{conjecture}

\noindent This is essentially the nicest result we can hope to get. For instance if $k=2$ and $\chi (G) =1$ then $G$ has no edges, so it cannot contain a $K_2$-minor.

\noindent Hadwiger's Conjecture is known to hold for the cases $k \leq 6$:

\begin{description}
\item [k = 2:] trivial; a graph with chromatic number larger than $1$ must have an edge, so it contains $K_2$ (as a subgraph, hence as a minor).
\item [k = 3, 4:] we proved these cases above.
\item [k = 5, 6:] these cases are implied by the Four Colour Theorem. See Robertson, Seymour, Thomas (1993).
\item [k > 6 :] higher cases remain unresolved.
\end{description}

\end{document}